\documentclass[11pt]{amsart}
\usepackage{enumerate,mathrsfs,amsthm,amsmath,amssymb}
	
\usepackage{color}  
\theoremstyle{break}
\newtheorem{thm}{Theorem}[section]
\newtheorem{prop}[thm]{Proposition}
\newtheorem{cor}[thm]{Corollary}
\newtheorem{lem}[thm]{Lemma}
\newtheorem{defn}[thm]{Definition}
\newtheorem{rk}[thm]{Remark}
\theoremstyle{remark}
\newtheorem{ex}[thm]{\bf Example}

\newcommand{\Max}{\mbox{{\rm\texttt{Max}}}}
\newcommand{\Min}{\mbox{{\rm\texttt{Min}}}}

\newcommand{\zar}{\mbox{{\rm\texttt{Zar}}}}
\newcommand{\msX}{{\mathscr{X}}}

\newcommand{\Cl}{\mbox{\it\texttt{Cl}}}
\newcommand{\ad}{\mbox{\it\texttt{Cl}}}
\newcommand{\Spec}{\mbox{{\rm\texttt{Spec}}}}

\newcommand{\T}{{\rm\texttt{T}}}
\newcommand{\Card}{\mbox{{\rm\texttt{Card}}}}
\newcommand{\sss}{\mbox{{\rm\texttt{Spec}}}}

        \newcommand{\Kr}{\mbox{{\rm\texttt{Kr}}}}%
         \newcommand{\kr}{\mbox{{\rm\texttt{kr}}}}
    \newcommand{\ms}{\mathscr}

    \newcommand{\z}{\ldots}

 \newcommand{\f}{\mathfrak}

\title[Closure operations in Zariski-Riemann spaces]{Some closure operations in Zariski-Riemann spaces of valuation domains: a survey}

\author{Carmelo A. Finocchiaro, \ Marco Fontana, \ and \ K. Alan Loper}
\address{C.F. \& M.F., Dipartimento di Matematica e Fisica, Universit\`a degli Studi
``Roma Tre'', Largo San Leonardo Murialdo, 1,  00146 Rome, Italy.}
\email{carmelo@mat.uniroma3.it }
\email{fontana@mat.uniroma3.it }
\address{K.A.L., Department of Mathematics, Ohio State University, Newark, OH 43055, USA}
\email{lopera@math.ohio-state.edu}

\thanks{\it Acknowledgments. \rm
During the preparation of this paper, the first two authors
were  partially supported by  a research grant PRIN-MiUR. 
}

\begin{document}

\maketitle

{\bf Abstract.} {In this survey we present several results concerning various topologies that were introduced  in recent years on spaces of valuation domains.}

\medskip


\section{Spaces of valuation domains}
The motivations for studying from a topological point of view spaces of valuation domains come from various directions and, historically, mainly from Zariski's work on the reduction of singularities of an algebraic surface and a three-dimensional variety and, more generally, for establishing new foundations of algebraic geometry by algebraic means (see \cite{za-39}, \cite{za}, \cite{za-44} and \cite{zasa}).    
Other important applications  with algebro-geometric flavour are due to Nagata \cite{na-62} \cite{na-63}, Temkin \cite{te-1}, Temkin  and Tyomkin \cite{tety}. 
Further motivations come from rigid algebraic geometry started by J. Tate \cite{ta} (see also   the  papers by Fresnel and van der Put \cite{fr-vdP},  Huber and Knebusch \cite{hu-kn},  Fujiwara and Kato \cite{fuka}), and from real algebraic geometry (see for instance Schwartz \cite{sch} and Huber \cite{hu-2}). For a deeper insight on these topics see \cite{hu-kn}. 

\bigskip

 In the following, we want to present some recent results in the literature concerning various topologies on collections of valuation domains.

Let $K$ be a field, $A$ an \emph{arbitrary} subring of $K$ and let ${\rm qf}(A)$ denote the quotient field of $A$. Set
$$
\zar(K|A):=\{V\mid V \mbox{ is a valuation domain and } A\subseteq V\subseteq K={\rm qf}(V)\}\,.
$$
When $A$ is the prime subring of $K$, we will simply denote by $\zar(K)$ the space $\zar(K|A)$. Recall that O. Zariski in \cite{za} introduced a topological structure on the set $Z:=\zar(K|A)$ by taking, as a basis for the open sets, the subsets $B_F:=B_F^Z:=\{V\in Z\mid V\supseteq F\}$, for $F$ varying in the fami\-ly 
of all finite subsets of $K$  (see also \cite[Chapter VI, \S 17, page 110]{zasa}).
When no confusion can arise, we will simply denote by $B_x$ the basic open set $B_{\{x\}}$ of $Z$. This topology is what that is now called \emph{the Zariski topology on $Z$} and the set $Z$, equipped with this topology, denoted also by $Z^{\mbox{\tiny{\texttt{zar}}}}$, is usually called  \emph{the  Zariski-Riemann space of $K$ over $A$} (sometimes called abstract Riemann surface or generalized Riemann manifold). 

 In 1944, Zariski \cite{za} proved the quasi-compactness of $Z^{\mbox{\tiny{\texttt{zar}}}}$ and later it was proven and rediscovered by several authors, with a variety of different techniques, that if $A$ is an integral domain and $K$ is the quotient field of $A$, then $Z^{\mbox{\tiny{\texttt{zar}}}}$ is a spectral space, in the sense of M. Hochster \cite{ho-69}.
More precisely, in 1986-87, Dobbs, Fedder and Fontana in \cite[Theorem 4.1]{dofefo-87} gave a purely topological proof of this fact and
Dobbs and Fontana presented a more complete version of this result in \cite[Theorem
2]{dofo-86} by exhibiting a ring $R$ (namely, the Kronecker function ring of the integral closure of $A$ with respect to
the $b$-operation) such that $Z$ is canonically homeomorphic
to $\sss( R)$ (both endowed with the Zariski topology). Later, using a general construction of the Kronecker function ring developed by F. Halter-Koch \cite{hk-03}, it was proved that the Zariski-Riemann space $Z$ is a still a spectral space when $K$ 
is not necessarily the quotient field of $A$ (see  \cite[Proposition 2.7]{heu-10} or \cite[Corollary 3.6]{fifolo-13b}). In  2004, in the appendix of \cite{ku-04},
Kuhlmann gave a model-theoretic proof of the fact that $Z$ is a spectral space. Note also that a purely topological approach for proving that $Z$ is spectral was presented by Finocchiaro in \cite[Corollary 3.3]{fi-13}.  Very recently, N. Schwartz \cite{sch-13}, using the inverse
spectrum of a lattice ordered abelian group and its structure sheaf  (see also Rump and Yang \cite{ruya-08}) obtained, as an application of his main theorem, (via the Jaffard-Ohm Theorem) a new proof of the fact that $Z$ is spectral.

Since 
$Z$ is a spectral space, $Z$ also possesses the constructible (or patch), the ultrafilter and the inverse topologies  (definitions will be recalled later) and these other topologies turn out to be more useful than the Zariski topology in several contexts as we will see in the present survey paper.

\section{The constructible topology}
Let $A$ be a ring and let $X:=\sss(A)$ denote the collection of all prime ideals of $A$. The set $X$  can be endowed with the \emph{Zariski topology} which has several attractive properties related to the ``geometric aspects" of the set of prime ideals. As  is well known, $X^{\mbox{\tiny{\texttt{zar}}}}$ (i.e., the set $X$ with the Zariski topology), is always quasi-compact, but almost never Hausdorff. More precisely, $X^{\mbox{\tiny{\texttt{zar}}}}$ is Hausdorff if and only if $\dim(A)=0$. Thus, many authors have considered a finer topology on the prime spectrum of a ring, known as \emph{the constructible topology} (see \cite{ch}, \cite{EGA}) or as \emph{the patch topology} \cite{ho-69}. 

In order to introduce this kind of topology in a more general setting, with a simple set-theoretical approach, we need some notation and terminology. Let $\ms X$ be a topological space. Following \cite{sch-tr}, we set
$$
\begin{array}{rl}
 \mathring{\mathcal K}\!(\ms X):=&\hskip -3pt  \{ U \mid U \subseteq \ms X,\, U \mbox{ open and quasi-compact in } \ms X\},\\
\overline{\mathcal K}\!(\ms X):=&\hskip -3pt  \{ \ms X \setminus U \mid U\in \mathring{\mathcal K}\!(\ms X) \},\\
{\mathcal K}\!(\ms X):=&\hskip -2pt  \mbox{the Boolean algebra of the subsets of $\ms X$ generated by $\mathring{\mathcal K}\!(\ms X)$}.
\end{array}
$$  
As in \cite{sch-tr}, we call \emph{the constructible topology on $\ms X$} the topology on $\ms X$ whose basis of open sets is $\mathcal K(\ms X)$. We denote by $\ms X^{\mbox{\tiny{\texttt{cons}}}}$ the set $\ms X$, equipped with the constructible topology. In particular, when $\ms X$ is a spectral space, the closure of a subset $Y$ of $\ms X$ under the constructible topology is given by:
$$
\begin{array}{rl}
\Cl^{\mbox{\tiny{\texttt{cons}}}}(Y) \! =  \!\bigcap \{U \!\cup\! (\ms X\!\setminus\! V) \mid  & \hskip -5pt \mbox{ $U$ and $V$ open and quasi-compact in }\ms X,  \\
 & \hskip -4pt U  \!\cup \! (\ms X \!\setminus \! V) \supseteq Y \}\,.
 \end{array}
$$
Note that,  for Noetherian topological spaces, this definition of constructible topology coincides with the classical one given in \cite{ch}. 

When $X:=\sss(A)$, for some ring $A$, then the set $\mathring{\mathcal K}\!(X^{\mbox{\tiny{\texttt{zar}}}})$ is a basis of open sets for $X^{\mbox{\tiny{\texttt{zar}}}}$, and thus the constructible topology on $X$ is finer than the Zariski topology. Moreover, $X^{\mbox{\tiny{\texttt{cons}}}}$ is a compact Hausdorff space and the constructible topology on $X$ is the coarsest topology for which $\mathring{\mathcal K}\!(X^{\mbox{\tiny{\texttt{zar}}}})$ is a collection of clopen sets (see \cite[{\bf I.}7.2]{EGA}).

\section{The ultrafilter topology}
In 2008,  the authors of \cite{folo}  considered ``another" natural topology on $X:=\sss(A)$, by using the notion of an ultrafilter and the following lemma.  
\begin{lem}{\rm (Cahen-Loper-Tartarone, \cite[Lemma 2.4]{calota})}
Let $Y$ be a subset of $X:=\sss(A)$ and let $\ms U$ be an ultrafilter on $Y$. Then $\f p_{\ms U}:=\{f\in A \mid V(f)\cap Y\in \ms U\}$ is a prime ideal of $A$ called \emph{the ultrafilter limit point of $Y$, with respect to $\ms U$}. 
\end{lem}
The notion of ultrafilter limit points of sets of prime ideals has been used to great effect in  several recent papers \cite{calota}, \cite{Loper1}, \cite{Loper2}.
If $\ms U$ is a trivial (or, principal) ultrafilter on the subset $Y$ of $X$, i.e., $\ms U=\{S\subseteq Y\mid \f p\in S\}$, for some $\f p\in Y$, then $\f p_{\ms U}=\f p$. On the other hand, when $\ms U$ is a nontrivial ultrafilter on $Y$, then it may happen that $\f p_{\ms U}$ { does not belong to} $Y$. This fact motivates the following definition.

\begin{defn} \rm
Let $A$ be a ring and $Y$ be a subset of $X:=\sss(A)$. We say that $Y$ is \emph{ultrafilter closed} if  $\f p_{\ms U}\in Y$, for each ultrafilter $\ms U$ on $Y$.

\end{defn}

It is not hard to see that, for each $Y \subseteq X$,
$$
\Cl^{\mbox{\tiny{\texttt{ultra}}}}(Y) := \{ \f p_{\ms U} \mid \ms U \mbox{ ultrafilter on } Y\}$$ 
satisfies the Kuratowski closure axioms and the set of all ultrafilter closed sets of $X$ is the family of  closed sets for a topology on $X$, called \emph{the ultrafilter topology on $X$}. We denote the set $X$ endowed with the ultrafilter topology by $X^{\mbox{\tiny{\texttt{ultra}}}}$. The main result of \cite{folo} is the following.

\begin{thm}\label{spec-ultra-cons}
\emph{(Fontana-Loper \cite[Theorem 8]{folo})} Let $A$ be a ring.  The ultrafilter topology  coincides with the constructible topology on the prime spectrum $\sss(A)$.  
\end{thm}

Taking as starting point the situation described above for  the prime spectrum of a ring, the next goal is to define an ultrafilter topology on the set $Z:=\zar(K|A)$ (where $K$ is a field and $A$ is a subring of $K$) that is finer than the Zariski topology. We start by recalling the following useful fact.

\begin{lem}{\rm (Cahen-Loper-Tartarone, \cite[Lemma 2.9]{calota})}
Let $K$ be a field and $A$ be a subring of $K$. if $Y$ is a nonempty subset of $Z:=\zar(K|A)$ and $\ms U$ is an ultrafilter on $Y$, then
$$
A_{\ms U,Y}:=A_{\ms U}:=\{x\in K\mid B_x\cap Y\in \ms U\}
$$
is a valuation domain of $K$ containing $A$ as a subring (i.e., $A_{\ms U}\in Z$), called \emph{the ultrafilter limit point of $Y$ in $Z$ , with respect to $\ms U$}. 
\end{lem}
As before let $Y$ be a nonempty subset of $Z:=\zar(K|A)$, when $V\in Y$ and $\ms U:=\{S\subseteq Y\mid V\in S\}$ is the trivial ultrafilter of $Y$ generated by $V$, then $A_{\ms U}=V$. But, in general, it is possible to construct nontrivial ultrafilters on $Y$ whose ultrafilter limit point are not elements of $Y$. This leads to the following definition. 

\begin{defn} \rm
Let $K$ be a field and $A$ be a subring of $K$. A subset $Y$ of $Z:=\zar(K|A)$ is \emph{ultrafilter closed} if $A_{\ms U}\in Y$, for any ultrafilter $\ms U$ on $Y$.

For every $Y \subseteq Z$, we set
$$
\Cl^{\mbox{\tiny{\texttt{ultra}}}}(Y) := \{A_{\ms U} \mid \ms U \mbox{ ultrafilter on } Y \}\,.
$$
\end{defn}

\begin{thm}\label{zar-ultra}
\emph{(Finocchiaro-Fontana-Loper \cite[Proposition 3.3, Theorems 3.4 and 3.9]{fifolo-13a})}
Let $K$ be a field, $A$ be a subring of $K$, and $Z:=\zar(K|A)$. The following statements hold.

\begin{enumerate}[\rm (1)]
\item $\Cl^{\mbox{\tiny{\texttt{ultra}}}}$ satisfies the Kuratowski closure axioms and so the ultrafilter closed sets of $Z$ are the closed sets for a topology, called \emph{the ultrafilter topology on $Z$}. 
\item Denote by $Z^{\mbox{\tiny{\texttt{ultra}}}}$ the set $Z$ equipped with the ultrafilter topology. Then, $Z^{\mbox{\tiny{\texttt{ultra}}}}$ is a compact Hausdorff topological space. 
\item The ultrafilter topology is the coarsest topology for which the basic open sets $B_F$ of the Zariski topology of $Z$ are clopen. In particular, the ultrafilter topology on $Z$ is finer than the Zariski topology and coincides with
the constructible topology.
\item The surjective map $\gamma: \zar(K|A)^{\mbox{\tiny{\texttt{ultra}}}}\rightarrow \sss(A)^{\mbox{\tiny{\texttt{ultra}}}}$, mapping a valuation domain to its center on $A$, is continuous and closed.
\item If $A$ is a Pr\"ufer domain, the map $\gamma:\zar(K|A)^{\mbox{\tiny{\texttt{ultra}}}}\rightarrow \sss(A)^{\mbox{\tiny{\texttt{ultra}}}}$ is a homeomorphism. 
\end{enumerate}

\end{thm}

\begin{rk} \rm (a)  From  Theorem \ref{spec-ultra-cons} and the last statement in point (3) of the previous theorem it is obvious that   points (4) and (5) of Theorem \ref{zar-ultra} hold  when one replaces everywhere ``${\mbox{{\texttt{ultra}}}}$''  with ``${\mbox{{\texttt{cons}}}}$''.

(b) It is well known that  points (4) and (5) of Theorem \ref{zar-ultra} hold  when both spaces are endowed with the Zariski topology (\cite[Theorem 2.5 and 4.1]{dofefo-87} and \cite[Theorem 2 and Remark 3]{dofo-86}).
\end{rk}

We recall now another important notion introduced by Halter--Koch in \cite{hk-03} as a generalization of the classical construction of the Kronecker function ring.

\begin{defn}
\rm
Let $T$ be an indeterminate over the field $K$. A subring $S$ of $K(T)$ is called \emph{a $K$--function ring}  if (a) $T$ and $T^{-1}$ belong to  $S$, and (b) ${f(0)}\in f(T)S$, for each nonzero polynomial $f(T)\in K[T]$. 
\end{defn}

We collect in the following proposition the basic algebraic properties of $K-$function rings \cite[Remarks at page 47 and Theorem (2.2)]{hk-03}.

\begin{prop}\label{basic} 
\emph{(Halter-Koch \cite[Section 2, Remark (1, 2 and 3), Theorem 2.2 and Corollary 2.7]{hk-03})} Let $K$ be a field, $T$ an indeterminate over $K$ and let  $S$ be a subring of $K(T)$. Assume that $S$  is a $K$--function ring.
\begin{enumerate}
\item[\rm (1)] If $S'$ is a subring of $K(T)$ containing $S$, then $S'$ is also a $K$--function ring.
\item[\rm (2)] If $\ms S$ is a nonempty collection of $K$--function rings (in $K(T)$), then $\bigcap \{ \Sigma \mid \Sigma \in\ms S\}$ is a $K$--function ring.
\item[\rm (3)] $S$ is a B\'ezout domain with quotient field $K(T)$.
\item[\rm (4)] If $f:=f_0+ f_1T+\z+f_rT^r\in K[T]$, then ${ (f_0, f_1,\z, f_r)S} =fS$.
\item[\rm (5)] For every valuation domain $V$ of $K$, the trivial extension or Gaussian extension $V(T)$ in 
$K(T)$ (i.e., $V(T) := V[T]_{M[T]}$, where $M$ is the maximal ideal of $V$) 
is a $K$--function ring. 
\end{enumerate}
\end{prop}

\bigskip

Given a subring $S$ of $K(T)$, we will denote by $\zar_0(K(T)|S)$ the subset of $\zar(K(T)|S)$ consisting of all the valuation domains of $K(T)$ that are trivial extensions of some valuation domain of $K$. 

The following characterization of $K$--function rings provides a slight gene\-ralization of \cite[Theorem 2.3]{heu-10} and its
proof is similar to that given by O. Kwegna Heubo, which is based on the work by
Halter-Koch  \cite{hk-03}.

\begin{prop}\label{zar0}
\emph{(Finocchiaro-Fontana-Loper \cite[Propositions 3.2 and 3.3]{fifolo-13a})} Let $K$ be a field, $T$  an indeterminate over $K$ and $S$  a subring of $K(T)$. Then, the following conditions are equivalent. 
\begin{enumerate}[\rm (i)]
\item $S$ is a $K-$function ring. 
\item $S$ is integrally closed and $\zar(K(T)|S)=\zar_0(K(T)|S)$. 
\item $S$ is the intersection of a nonempty subcollection of $\zar_0(K(T))$. 
\end{enumerate}
\end{prop}

\smallskip

We give next one of the main results in \cite{fifolo-13a} which, for the case of the Zariski to\-po\-logy, was already proved   in \cite[Corollary 2.2, Proposition 2.7 and Corollary 2.9]{heu-10}. More precisely,

\begin{thm} \label{zar-spectral}
\emph{(Finocchiaro-Fontana-Loper \cite[Corollary 3.6, Proposition 3.9, Corollary 3.11]{fifolo-13a})} 
Let $K$ be a field and $T$  an indeterminate over $K$. The following statements hold.
\begin{enumerate}[\rm(1)]
\item The natural map $\varphi:\zar(K(T))\rightarrow \zar(K)$, $W\mapsto W\cap K$, is continuous and closed  with respect to both the Zariski topology and the ultrafilter topology (on both spaces).  
\item If $S\subseteq K(T)$ is a $K$--function ring, then the restriction of $\varphi$ to the subspace $\zar(K(T)|S)$ of $\zar(K(T))$ is a topological embedding, with respect to   both  the Zariski topology and the ultrafilter topology. 
\item Let $A$ be \emph{any} subring of $K$, and let $${\rm Kr}(K|A):=\bigcap \{V(T)\mid V\in \zar(K|A)\}.$$ Then ${\rm Kr}(K|A)$ is a $K$--function ring. Moreover, the restriction of the map $\varphi$ to $\zar(K(T)|{\rm Kr}(K|A))$ establishes a homeomorphism of \  $\zar(K(T)|{\rm Kr}(K|A))$ with $\zar(K|A)$, with respect   to both the   Zariski topology and the  ultrafilter topology.
\item Let $A$ be a subring of $K$, $S_A:={\rm Kr}(K|A)$, and let $\gamma:\zar(K(T)|S_A)\rightarrow \sss(S_A)$ be the map sending a valuation overring of $S_A$ into its center on $S_A$. Then $\gamma$ establishes a homeomorphism, with respect   to both the Zariski topology and  the ultrafilter topology; thus, the map $$\sigma:= \gamma\circ \varphi^{-1}:\zar(K|A)\rightarrow \zar(K(T)|S_A)\rightarrow \sss(S_A)
$$
is also a homeomorphism. In other words, $\zar(K|A)$ is a spectral space when endowed with either the Zariski topology or the ultrafilter to\-po\-logy.
\end{enumerate}
\end{thm}

Note that  statement (4) of the previous theorem extends \cite[Theorem
2]{dofo-86} to the general case where $A$ is an arbitrary subring of the field $K$.

\smallskip

\section{The inverse topology}

Let $\ms X$ be  any topological space. Then, it is well known that the topology induces a natural {\sl preorder} on $\ms X$ by setting
$$
x\leq y \ :\Leftrightarrow \  y\in \Cl(\{x\}).
$$
Therefore:
$$
x^\uparrow :=\{ y \in \ms X \mid  x \leq y \} =  \Cl(\{x\});
$$
in particular, if $F$ is a closed subspace of ${\mathscr{X}}$ and $x \in F$, then $x^\uparrow \subseteq F$.\\
The set $x^\uparrow$ is called the set of \emph{specializations of} $x$ in ${\mathscr{X}}$; on the other hand, the set 
$$
x^\downarrow :=\{ y \in \ms X \mid  y \leq x \} 
$$
is called the set of \emph{generizations of} $x$.  Since the closed subspaces are closed under specializations, it follows easily that   if $U$ is an open subspace of ${\mathscr{X}}$ and $x \in U$, then $x^\downarrow \subseteq U$.

For a subset $Y$ of $\ms X$ we denote by $Y^\uparrow$ (respectively, $Y^\downarrow$) the set of all specializations (respectively, generizations) of elements in $Y$.

 If $\ms X$ is a \T$_0$-space, then the preorder is a partial order on $\ms X$ and, for $x, y \in\ms X$,  $x^\uparrow= 
y^\uparrow$ if and only if $x =y$.

\medskip

Given a preordered set $(X, \leq)$, we say that {\it a topology $\ms T$ on $X$} is {\it compa\-ti\-ble with the order $\leq$} \ if, for each pair of elements $x$ and $y$ in $X$, $y \in \Cl^{\ms T}(x)$ implies that $x \leq y$. Obviously, in general, several different topologies on $X$ may be compatible with the given order on $X$.

The following properties are easy consequences of the definitions (see, for instance, \cite[Lemma 2.1, Proposition 2.3(b)]{dofopa-80}).

\begin{lem} \label{L-R}  Let $(X, \leq)$ be a preordered set and let $Y \subseteq X$.
\begin{enumerate}[\rm (1)] 
\item $\Cl^L(Y) := Y^\uparrow$ (respectively, $\Cl^R(Y) := Y^\downarrow$) satisfies the  Kuratowski closure axioms and so it defines a topological structure on $
X$, called the \emph{L(eft)-topology} (respectively, the \emph{R(ight)-topology}) on $X$.
\item The L-topology (respectively, R-topology) on $X$ is the finest topology on $X$ compatible with the given order (respectively, with the opposite order of the given order) on $X$.
\item A subset $U$ of $X$ is open in the L-topology (respectively, R-topology) if and only if $U = U^\downarrow \ (= \Cl^R(U), \mbox{i.e., it is closed in the R-topology})$ (respectively, $U = U^\uparrow \ (= \Cl^L(U), \mbox{i.e., it is closed in the L-topology})$.
\item Let $U \subseteq X$ be a  nonempy open subspace of $X$ endowed with the L-topology (respectively, R-topology). Then $U$ is quasi-compact if and only if there exist $x_1, x_2, \dots x_n$ in $U$, with $n \geq 1$, such that $U = x_1^\downarrow \cup x_2^\downarrow \cup \dots  \cup x_n^\downarrow$ (respectively, $U = x_1^\uparrow \cup x_2^\uparrow \cup \dots  \cup x_n^\uparrow$ ).

\end{enumerate}
\end{lem}

\begin{rk} \rm In relation with Lemma \ref{L-R}(2), note that the COP (or, Closure Of Points) topology \cite{leoh-76} is the coarsest topology on $X$ compatible with a given order on $X$. 
\end{rk}

Recall that a topological space $\ms X$ is an {\it Alexandroff-discrete  space} if it is  \T$_0$ and for each subset $Y $ of $\ms X$ the closure of $Y$ coincides with the union of the closures of its points \cite[page 28]{al-56}.  Therefore, if $(X, \leq)$ is a partially ordered set, then the L-topology (or the R-topology) determines on $X$ the structure of  an Alexandroff-discrete space.

\medskip
If $\ms X$ is a \T$_0$ topological space, then the L-topology on $\ms X$, associated to the partial order defined by the given topology on $\ms X$, is finer than the original topology of $\ms X$, since for each $Y \subseteq \ms X$, $\Cl^L(Y) \subseteq \Cl(Y)$.  Moreover, even if $\ms X$  is a spectral space,  $\ms X^L$ (i.e., $\ms X$ equipped with the L-topology) is not spectral in general. For example, a spectral space having infinitely many closed points may not be quasi-compact with respect to the L-topology by Lemma \ref{L-R}(4) (e.g., $\ms X := \Spec(\mathbb Z) = \bigcup \{(p )^\downarrow \mid (p ) \mbox{ is a nonzero prime ideal of } \mathbb Z\}$ is an open cover of $\ms X$ endowed with the L-topology without a finite open subcover). 
\medskip

Before stating a result providing  a complete answer to the question of when the L-topology determines  a spectral space (see Theorem \ref{th-dolopa}),  we recall a useful application of the L-topology
showing that the constructible closure and the closure by specializations  (or, L-closure) determines the structural closure in a spectral space (see, for instance,  \cite[Corollary to Theorem 1]{ho-69}, \cite[Lemma 1.1]{fo-80} or \cite[Proposition 3.1(a)]{dofopa-80}).

\begin{lem} \label{special}
Let $\ms X$ be a spectral space. For each subset $Y$ of $\ms X$, 
$$ \Cl(Y) = \Cl^L(\Cl^{\mbox{\tiny{\texttt{cons}}}}(Y))\,.$$
\end{lem}

\medskip

Let $(X, \leq)$ be a preordered set and denote by $ \Max(X)$ (respectively, $ \Min(X)$) the set of all maximal (respectively,  minimal) elements of $X$. In particular, if  
$\ms X$ is a topological space, we denote by $\Max(\ms X)$ (respectively, $\Min(\ms X)$ the set of all maximal (respectively,  minimal) points of a topological space $\ms X$, with respect to the preorder $\leq$ induced by the topology of $\ms X$. It follows immediately by definition that 
$$x\in \Max(\ms X) \; \Leftrightarrow  \;\{x\} \mbox{ is closed in } \ms X \; \Leftrightarrow  \;\{x\} \mbox{ is closed in } \ms X^L \,.$$
From the order-theoretic point of view, we have the following.

\begin{lem}\label{min-max}  Let $(X, \leq)$ be a  partially ordered set. 
\begin{enumerate}
\item[\emph{(1)}] The following conditions are equivalent:
{\begin{enumerate}
\item[\emph{(i)}] $x$ is a closed point in $X^L$;
\item[\emph{(ii)}] $x \in \Max(X)$;
\item[\emph{(iii)}]  $x$ is an open point in $X^R$.
\end{enumerate}}

\item[\emph{(2)}] The following conditions are equivalent:
{\begin{enumerate}
\item[\emph{(i)}] $x$ is an open point in $X^L$;
\item[\emph{(ii)}] $x \in \Min(X)$;
\item[\emph{(iii)}]  $x$ is a closed point in $X^R$.
\end{enumerate}}
\item[\emph{(3)}] $X^L$ (respectively, $X^R$) is a \T$_1$ topological space if and only if $X^L$ (respectively, $X^R$) 
is a discrete space.
\end{enumerate}
\end{lem}

\medskip

\begin{thm}\label{th-dolopa}
 \emph{(Dobbs-Fontana-Papick \cite[Theorem 2.4]{dofopa-80})} Let  $X$  be  a partially  ordered  set.  Then  $X$  with the  L-topology is a spectral space  if  and  only 
if the   following  four  properties  hold: 
\begin{enumerate}
\item[\emph{($\alpha$)}]
Each  nonempty  totally  ordered  subset $Y$ of   $X$  has  a  $\sup$; 

\item[\emph{($\beta$)}]  X  satisfies  the  following condition:
{\begin{enumerate}
\item[\emph{(filtr$^L$)}]  each  nonempty  lower-directed  subset  $Y$  of  $X$  has  a  greatest  lower 
bound  $y := \inf (Y)$  such  that $ y^\uparrow =  Y^\uparrow$; 
\end{enumerate}}
\item[\emph{($\gamma$)}]  $\Card(\Max (X))$ is finite;  
\item[\emph{($\delta$)}] For  each  pair  of distinct  elements  $x$  and  $y$  of  $X$,  there  exist  at  most  finitely 
many  elements  of  $X$  which  are  maximal  in  the  set  of  common  lower  bounds 
of  $x$  and  $y$. 
\end{enumerate}

\end{thm}

The necessity of condition  ($\alpha$)  follows from \cite[Theorem 9]{ka-74}, that of condition  ($\beta$) uses \cite[Proposition 5]{ho-69} and the necessity of condition  ($\delta$) is related to Lemma \ref{L-R}(4); condition  ($\gamma$) holds in any  Alexandroff-discrete space. The sufficiency  of ($\alpha$)--($\delta$) results by verifying the conditions of Hochster's characterization theorem \cite{ho-69}.

\medskip

Using the opposite order,  from Theorem \ref{th-dolopa} we can  easily deduce a characterization of when  a partially  ordered  set with the  R-topology is a spectral space. 
\medskip

Given a spectral space $\ms X$, the following proposition gives a complete answer to the question of when the continuous map $\ms X^L \rightarrow \ms X$ (where $\ms X^L $ denotes the topological space $\ms X$ equipped with the L-topology, associated to the partial order defined by the given topology on $\ms X$) is a homeomorphism. In particular, in this situation,  $\ms X^L$ is a spectral space. 

\begin{prop}\label{X=XL}
 \emph{(Picavet \cite[V, Proposition 1]{pi-75},  Dobbs-Fontana-Papick \cite[Theorem 3.3]{dofopa-80})} Let $\ms X$ be a spectral space. The following are equivalent.
\begin{enumerate}
\item[\emph{(i)}]  $\ms X^L = \ms X$.

\item[\emph{(ii)}] For each $x \in X$, $x^\downarrow$ is a quasi-compact open subset of $\ms X$.

\item[\emph{(iii)}] For each family $\{U_\lambda \mid \lambda \in \Lambda\} $ of quasi-compact open subsets of  $\ms X$, the set 
 $\bigcap \{U_\lambda \mid \lambda \in \Lambda\} $ is still a quasi-compact open subset of $\ms X$.

\item[\emph{(iv)}]  Each increasing sequence of irreducible closed subsets of  $\ms X $ stabilizes 
and, \ for each family $\{U_\lambda \mid \lambda \in \Lambda\} $ of quasi-compact open subsets of the space \  $\ms X$, \newline
$\Card(\Max(\bigcap \{U_\lambda \mid \lambda \in \Lambda\}))$ is finite.
\end{enumerate}
\end{prop}

The previous proposition gives the motivation for studying the rings $A$ such that, for each $P \in \Spec(A)$, the canonical map $\Spec(A_P) \hookrightarrow \Spec(A)$ is open. This class of rings was introduced by G. Picavet in 1975 (\cite{pi-75}  and \cite{pi-75a}) under the name of \emph{g-ring} and it can be shown that  a spectral space $\ms X$ is such that $\ms X^L = \ms X$ if and only if $\ms X$ is homeomorphic to the prime spectrum of a g-ring \cite[V, Proposition 1]{pi-75}.

\bigskip

 When $\ms X$ is a spectral space, Hochster \cite{ho-69} introduced a new topology on $\ms X$, called the inverse topology. If we denote by  $\ms X^{\mbox{\tiny{\texttt{inv}}}}$, the set $\ms X$ equipped with the inverse topology, Hochster proved that $\ms X^{\mbox{\tiny{\texttt{inv}}}}$ is still a spectral space and the partial order on $\ms X$ induced by the inverse topology is the opposite order of that induced by the given topology on $\ms X$. More precisely:
 
 \begin{prop} \label{inv} 
 \emph{(Hochster \cite[Proposition 8]{ho-69})} Let $\ms X$ be a spectral space. For each subset $Y$ of $\ms X$, set:
 $$
 \Cl^{\mbox{\tiny{\texttt{inv}}}}(Y) :=
\bigcap \{U   \mid  \mbox{ $U$ open and quasi-compact in }\ms X,  \; 
U  \supseteq Y \}\,.
$$
\begin{enumerate}[\rm (1)]
\item $\Cl^{\mbox{\tiny{\texttt{inv}}}}$ satisfies the  Kuratowski closure axioms and so it defines a topological structure on $\ms X$, called the \emph{inverse topology}; denote by $\ms X^{\mbox{\tiny{\texttt{inv}}}}$  the set $\ms X$ equipped with the inverse topology.
\item   The partial order on $\ms X$ induced by the inverse topology is the opposite order of that induced by the given topology on $\ms X$.
\item $\ms X^{\mbox{\tiny{\texttt{inv}}}}$ is a spectral space.
\end{enumerate}
  \end{prop}
  
  \medskip
  
  Let $\ms X$ be a spectral space. For each subset $Y$ of $\ms X$, set:
 $$
\Max(\Cl^{\mbox{\tiny{\texttt{inv}}}}(Y)):= \{x \in \Cl^{\mbox{\tiny{\texttt{inv}}}}(Y) \mid x^\uparrow  \cap \Cl^{\mbox{\tiny{\texttt{inv}}}}(Y)  = \{x\}\} \,.$$

B. Olberding in \cite[Proposition 2.1(2)]{ol-13} has observed that:

$$
\Max(\Cl^{\mbox{\tiny{\texttt{inv}}}}(Y)) \subseteq \Cl^{\mbox{\tiny{\texttt{cons}}}}(Y)\,.$$

\medskip

  From the previous observation and from Lemmas \ref{L-R},  \ref{special} and \ref{min-max}   and from Proposition \ref{inv},  it is not very hard to prove the following (see also \cite[Section 3]{dofopa-80}, \cite[Remark 2.2 and Proposition 2.3]{sch-tr},  \cite[Proposition 2.3]{ol-13} and \cite[Proposition 2.6]{fifolo-13b}).

  \begin{cor} \label{R-inv} Let $\ms X$ be a spectral space. 
  \begin{enumerate}[\rm (1)]
  \item The constructible topology on $\ms X^{\mbox{\tiny{\texttt{inv}}}}$ 
  coincides with the constructible topology on $\ms X$, i.e., $(\ms X^{\mbox{\tiny{\texttt{inv}}}})^{\mbox{\tiny{\texttt{cons}}}} =\ms X^{\mbox{\tiny{\texttt{cons}}}} $.
  \item For each subset $Y$ of $\ms X$, 
  $$\Cl^{\mbox{\tiny{\texttt{inv}}}}(Y) =  \Cl^R( \Max(\Cl^{\mbox{\tiny{\texttt{inv}}}}(Y)))=
 \Cl^R(\Cl^{\mbox{\tiny{\texttt{cons}}}}(Y))\}\,.
 $$
 \item
 $(\ms X^{\mbox{\tiny{\texttt{inv}}}})^L = \ms X^R$ and $(\ms X^{\mbox{\tiny{\texttt{inv}}}})^R = \ms X^L$.
  \item $(\ms X^{\mbox{\tiny{\texttt{inv}}}})^{\mbox{\tiny{\texttt{inv}}}} = \ms X$.
 \item For each $x \in \ms X$, $\Cl^{\mbox{\tiny{\texttt{inv}}}}(x) = \Cl^R(x)$ is an irreducible closed set of $\ms X^{\mbox{\tiny{\texttt{inv}}}}$ (and, obviously,  $\Cl(x) = \Cl^L(x)$ is an irreducible closed set of $\ms X$).
  \item The topological space $\ms X$ is irreducible if and only if $\ms X^{\mbox{\tiny{\texttt{inv}}}}$ has a unique closed point.
  
  \item The topological space $\ms X$ has a unique closed point if and only if $\ms X^{\mbox{\tiny{\texttt{inv}}}}$ is irreducible.
  
  \item The following are equivalent.

\begin{enumerate}
\item[\emph{(i)}] $Y$ is quasi-compact in $\ms X$.
\item[\emph{(ii)}] $\Cl^R(Y)$ is quasi-compact in $\ms X$.
\item[\emph{(iii)}] $\Cl^{\mbox{\tiny{\texttt{inv}}}}(Y)= \Cl^R(Y)$. 
\item[\emph{(iv)}] $\Cl^{\mbox{\tiny{\texttt{cons}}}}(\Cl^R(Y)) = \Cl^R(Y) $. 
\item[\emph{(v)}]$\Max(\Cl^{\mbox{\tiny{\texttt{inv}}}}(Y)) \subseteq Y$.
  \end{enumerate}

\end{enumerate}
 \end{cor}

\medskip
  By using the inverse topology, we can state an easy corollary of Proposition \ref{X=XL} and Corollary \ref{R-inv} (see also \cite[Theorem 3.3, Corollaries 3.4 and 3.5]{dofopa-80}),  which provides in part (1) further characterizations of when $\ms X^L = \ms X$.
  
  \begin{cor} Let $\ms X$ be a spectral space. 
  \begin{enumerate}
  \item[\emph{(1)}] 
  The following are equivalent.
  {
\begin{enumerate}
\item[\emph{(i)}]  $\ms X^L = \ms X$  (i.e., $\ms X$ is an Alexandroff-discrete topological space).

 \item[\emph{(ii)}] Each open subset of $\ms X^{\mbox{\tiny{\texttt{inv}}}}$ is the complement of a quasi-compact open subset
of $\ms X$.
 
 \item[\emph{(iii)}]  $\ms X^{\mbox{\tiny{\texttt{inv}}}}$ is a Noetherian space. 
 \end{enumerate}
 }

 \item[\emph{(2)}] 
  The following are equivalent.
  {
\begin{enumerate}
\item[\emph{(i)}]  $\ms X^R = \ms X^{\mbox{\tiny{\texttt{inv}}}}$ (i.e., $\ms X^{\mbox{\tiny{\texttt{inv}}}}$ is an Alexandroff-discrete topological space).
 \item[\emph{(ii)}]   Each open subset of $\ms X$ is the complement of a quasi-compact open subset
of $\ms X^{\mbox{\tiny{\texttt{inv}}}}$ (or, equivalently,    each open subset of $\ms X$ is quasi-compact).
 \item[\emph{(iii)}]   $\ms X$ is an Noetherian topological space.
 \end{enumerate}
 } 
 
 \item[\emph{(3)}] 
  The following are equivalent.
  {
\begin{enumerate}
\item[\emph{(i)}]  $\ms X$ is a Noetherian Alexandroff-discrete space.
\item[\emph{(ii)}]  $\ms X^{\mbox{\tiny{\texttt{inv}}}}$ is a Noetherian Alexandroff-discrete space.

 \item[\emph{(iii)}]   $\Card(\ms X)$ is finite.
 \end{enumerate}
 }  
 
 \end{enumerate}
 
  \end{cor}
  
\bigskip

 Recall that  \emph{a spectral map} of spectral spaces $f: \ms X\rightarrow \ms Y$ is  a continuous map such that the preimage of every open and quasi-compact subset of $\ms Y$ under $f$ is again quasi-compact. We say that a spectral map of spectral spaces $f: \ms X\rightarrow \ms Y$
  is \textit{a going-down map} (respectively, \textit{a going-up map})  if, for any  pair of distinct elements $y', y \in \ms Y$
  such that $y'\in \{y\}^{\downarrow}$  (respectively, $y'\in \{y\}^{\uparrow}$
  and for any $x\in \ms X$ such that $f(x)=y$ there exists a point $x'\in \{x\}^{\downarrow}$ (respectively, $x'\in \{x\}^{\uparrow}$) such that $f(x')=y'$. 

\begin{lem} Let  $f: \ms X\rightarrow \ms Y$ be a spectral map of spectral spaces.
\begin{enumerate}
\item[\emph{(1)}] $f: \ms X^{\mbox{\tiny{\texttt{cons}}}}\rightarrow \ms Y^{\mbox{\tiny{\texttt{cons}}}}$ is a closed spectral map.
\item[\emph{(2)}] The following are equivalent
{\begin{enumerate}
\item[\emph{(i)}] $f$ is a going-down (respectively, going-up) map.

\item[\emph{(ii)}] $f(x^\downarrow) = f(x)^\downarrow$ (respectively, $f(x^\uparrow) = f(x)^\uparrow$), for each $x \in \ms X$

\item[\emph{(iii)}] $f({\ms X'}^\downarrow) = f({\ms X'})^\downarrow$ (respectively, $f({\ms X'}^\uparrow) = f({\ms X'})^\uparrow$,) for each $\ms X' \subseteq \msX$.

\item[\emph{(iv)}] The continuous map $f: \ms X^{R}\rightarrow \ms Y^{R}$ (respectively, $f: \ms X^{L}\rightarrow \ms Y^{L}$) is closed.
\end{enumerate}
}
\item[\emph{(3)}] $f: \ms X^{R}\rightarrow \ms Y^{R}$  is  closed  (respectively, open) if and only if $f: \ms X^{L}\rightarrow \ms Y^{L}$ is open (respectively, closed).
\item[\emph{(4)}]  If $f: \ms X\rightarrow \ms Y$ is an open (respectively, closed) spectral map of spectral spaces then $f$ is a going-down (respectively, going-up) map.
\item[\emph{(5)}]  If $f: \ms X\rightarrow \ms Y$ is an open  spectral map of spectral spaces then $f: \ms X^{\mbox{\tiny{\texttt{inv}}}}\rightarrow \ms Y^{\mbox{\tiny{\texttt{inv}}}}$
 is a closed  spectral map.
\end{enumerate}
\end{lem}

\begin{proof} (1) is an obvious consequence of the definitions. (2)  Since a spectral map $f$ is continuous then it is straightforward that $x' \leq x$ in $\ms X$ implies that $f(x') \leq f(x)$ and so 
  $ f(x^{\downarrow})\subseteq f(x)^{\downarrow} $. 
  Moreover, for each $y \in \ms Y$, $f^{-1}(y^\downarrow) = \bigcup \{ x^\downarrow \mid x \in \ms X  \mbox{ and } f(x) \leq y \}$
   and so $f: \ms X^R \rightarrow \ms Y^R$ is also continuous.
   The various equivalences are now straightforward consequences of the definitions.
   
   (3) is an easy consequence of (2).
   
  (4)  Let $x \in \ms X$. It is easy to see that  $x^{\downarrow} =  \bigcap\{U \mid  U $  open and quasi-compact and  $ x \in U \subseteq \ms X\}$. Therefore, for any spectral map of spectral spaces $f: \ms X\rightarrow \ms Y$ and any $x\in \ms X$, the following holds:
$$
\begin{array}{rl}
f(x^{\downarrow}) & \hskip -2pt = f( \bigcap\{U \mid  U \mbox{ open and quasi-compact and } x \in U \subseteq \ms X\} ) \\  & \hskip -2pt \subseteq \bigcap\{f(U) \mid  U \mbox{ open and quasi-compact and } x \in U \subseteq \ms X \}\,.
\end{array}
$$
Conversely, assume that $f$ is an   open spectral map and take a point $y\in f(U)$, for any open and quasi-compact neighborhood $U$ of $x \in \ms X$. Consider the following collection of subsets of $\ms X$:
$$
\mathcal F:= \mathcal F(y) := \{f^{-1}(\{y\})\cap U \mid U \mbox{ open and quasi-compact and } x \in U \subseteq \ms X\}\,.
$$
Note now that $\mathcal F$ is obviously closed under finite intersections, since the quasi-compact open sets of $\ms X$ are closed under finite intersections and, by assumption, each set belonging to $\mathcal F$ is nonempty. On the other hand, the set $f^{-1}(\{y\})$ is closed  with respect to the constructible topology  on $\ms X$ and thus  is compact in $\ms X^{\mbox{\tiny{\texttt{cons}}}}$. 
Keeping in mind that each open and quasi-compact subspace of the given spectral topology on $\ms X$ is clopen in $\ms X^{\mbox{\tiny{\texttt{cons}}}}$, it follows immediately that $\mathcal F$ is a collection of closed subsets of the compact space $f^{-1}(\{y\}) \ (\subseteq \ms X^{\mbox{\tiny{\texttt{cons}}}})$, satisfying the finite intersection property.  Therefore, by compactness, there exists a point $x'\in f^{-1}(\{y\}) \cap U$, for any open and quasi-compact neighborhood $U$ of $x \in \ms X$. In particular,  $x' \in \bigcap \{U \mid U \mbox{ open and quasi-compact and } x \in U \subseteq \ms X\}= \{x\}^{\downarrow}$ and so $x' \leq x$. Therefore, $f(x') =y \leq f(x)$. We conclude that $ f(x^{\downarrow}) = \bigcap\{f(U) \mid  U \mbox{ open and quasi-compact and } x \in U \subseteq \ms X \}$.

On the other hand, since $f$ is open, we have:
$$
\begin{array}{rl}
f(x)^{\downarrow} & \hskip -2pt = \bigcap\{V \mid V \mbox{ open and quasi-compact and } f(x) \in V \subseteq \ms Y\}  \\  &\hskip -2pt  \subseteq \bigcap\{f(U) \mid  U \mbox{ open and quasi-compact and } x \in U \subseteq \ms X \} \\
&\hskip -2pt  = f(x^{\downarrow})\,.
\end{array}
$$
Since the opposite inclusion holds in general, we have  $f(x^\downarrow) = f(x)^\downarrow$ and so $f$ is a going-down spectral map.
 
The parenthetical statement is easier to prove. Indeed, suppose that $f$ is a closed spectral map,
 let  $y', y \in \ms Y$ be
  such that  $y'\in \{y\}^{\uparrow}$
  and let $x\in \ms X$ be such that $f(x)=y$.
By assumption, we have $f(\Cl(\{x\}))=\Cl(f(\{x\}))=\Cl(\{y\})$ and thus, since $y'\in \Cl(\{y\})$, there is a point $x'\in \Cl(\{x \})$ such that $f(x')=y'$. This shows that $f$ is a going-up map.

(5) If $f$ is an open spectral map then, by (3), $f$ is  going-down and thus, by (2), $f: \ms X^{R}\rightarrow \ms Y^{R}$ is  closed. Therefore, by using (1), for each $\ms X' \subseteq \ms X$, we have
$
\Cl^{\mbox{\tiny{\texttt{inv}}}}(\ms X') = 
\Cl^{R}(\Cl^{\mbox{\tiny{\texttt{cons}}}}(\ms X')) 
$ and so $f(\Cl^{\mbox{\tiny{\texttt{inv}}}}(\ms X')) = 
\Cl^{R}(f(\Cl^{\mbox{\tiny{\texttt{cons}}}}(\ms X'))) = 
\Cl^{R}(\Cl^{\mbox{\tiny{\texttt{cons}}}}(f(\ms X')))= \Cl^{\mbox{\tiny{\texttt{inv}}}}(f(\ms X'))$. 
\end{proof}

\begin{ex}
We now show that it is not true that
  if $f: \ms X\rightarrow \ms Y$ is a closed  spectral map of spectral spaces 
  then 
   $f: \ms X^{\mbox{\tiny{\texttt{inv}}}}\rightarrow \ms Y^{\mbox{\tiny{\texttt{inv}}}}$
 is an open  spectral map.
  As a matter of fact, let $K$ be a field and let $\boldsymbol{\mathcal T}:=\{T_i \mid i\in \mathbb N\}$ be an infinite and  countable collection of indeterminates over $K$. Let  $A:=K[\boldsymbol{\mathcal T}]$, let $M$ be the maximal ideal of $A$ generated by all the  indeterminates and let $B:= A/M$. Set $\ms X:=\Spec(B)$ and $\ms Y:=\Spec(A)$. Of course, the inclusion $f:\ms X\rightarrow \ms Y$  (associated to the canonical projection $A \rightarrow B$) is a closed embedding, with respect to the Zariski topology.
   We claim that $f$ is not open, if $\ms X$ and $\ms Y$ are endowed with the inverse topology. By contradiction, assume that $f: \ms X^{\mbox{\tiny{\texttt{inv}}}} \rightarrow \ms Y^{\mbox{\tiny{\texttt{inv}}}}$ is open.  In this situation, $\ms X$  should be open in $\ms Y^{\mbox{\tiny{\texttt{inv}}}}$, (since $\ms X$ is trivially open in $ \ms X^{\mbox{\tiny{\texttt{inv}}}}$).
    This implies that $\ms Z:=\ms Y\setminus f(\ms X) =\Spec(A)\setminus\{M\}$ is closed in $\ms Y^{\mbox{\tiny{\texttt{inv}}}}$, i.e., $\ms Z$ is an intersection of a family of open and quasi-compact subspaces of $\ms Y$. Since $\ms Z$ differs from $\ms Y$ for exactly one point, it has to be quasi-compact, with respect to the Zariski topology of $\ms Y$. On the other hand, it is immediately verified that the open cover 
$$\{\{P \in \ms Y\mid T_i\notin P\} \mid i \in \mathbb N\}$$
 of $\ms Z$ has no finite subcovers, a contradiction.  
\end{ex}

\section{Some applications}
The first application that we give is a topological interpretation of when two given collections of valuation domains are representations of the same integral domain.

\begin{prop}\label{closure-intersection}
\emph{(Finocchiaro-Fontana-Loper  \cite[Proposition 4.1]{fifolo-13b})}
Let $K$ be a field. If $Y_1,Y_2$ are nonempty subsets of \ $\zar(K)$ having  the same closure in  $\zar(K)$, with respect to the ultrafilter topology, then 
$$
\bigcap\{V\mid V\in Y_1\}=\bigcap \{V\mid V\in Y_2\}.
$$
In particular, 
$$
\bigcap\{V\mid V\in Y\}=\bigcap\{V\mid V\in \ad^{\mbox{\tiny{\texttt{ultra}}}}(Y)\}.
$$
\end{prop}

The converse of the   first statement in Proposition \ref{closure-intersection} is false (for an explicit example see Example 4.4 in \cite{fifolo-13b}). More precisely, we will show that equality of the closures of the subsets $Y_1,Y_2$, with respect to the ultrafilter topology,  implies a statement that, in general, is stronger than the equality of the (integrally closed) domains obtained by intersections. To see this, recall some background material about semistar operations. 

Let $A$ be an integral domain, and let $K$ be the quotient field of $A$. As usual, denote by $\overline{\boldsymbol F}(A)$ the set of all nonzero $A-$submodules of $K$, and by ${\boldsymbol f}(A)$ the set of all nonzero finitely generated $A-$submodules of $K$.  As  is well known, a nonempty subset $Y$ of $\zar(K|A)$ induces the \emph{valuative semistar operation} $\wedge_Y$, defined by 
$F^{\wedge_Y}:=\bigcap \{FV\mid V\in Y\}$, for each $F\in\overline{\boldsymbol F}(A)$.  A valuative semistar operation $\star$ is always \texttt{e.a.b.}, that is, for all $F,G,H\in {\boldsymbol f}(A)$, $(FG)^\star\subseteq (FH)^\star$ implies $G^\star\subseteq H^\star$  (for more details, see for example \cite{folo-07}). Recall that we can associate to  any semistar operation $\star$ on $A$ a semistar operation $\star_f$ \emph{of finite type} (on $A$), by setting $F^{\star_f}:=\bigcup\{G^\star\mid G\in {\boldsymbol f}(A), \ G\subseteq F\}$, for each $F\in\overline{\boldsymbol F}(A)$; $\star_f$ is called \emph{the semistar operation of finite type associated to $\star$}. 

Since, for each $V\in  \zar(K|A)$ (equipped with the classical Zariski topology), $\Cl(\{V\}) = \{ W \in \zar(K|A) \mid W\subseteq V \}$  \cite[Ch. VI, Theorem 38]{zasa}, the partial order associated to the Zariski topology of $\zar(K|A)$ is defined as follows:
$$
W  \preceq V \; :\Leftrightarrow \; V\subseteq W\,.
$$
 For any subset $Y\subseteq \zar(K|A)$, denote by $Y^\downharpoonright$ \emph{the Zariski--generic closure of $Y$}, that is,  $Y^\downharpoonright:=\{W\in \zar(K|A)\mid V\subseteq W, \mbox{ for some } V\in Y\} = \ad^{R}(Y)$. 
 It is obvious that $\wedge_Y = \wedge_{Y^{\downharpoonright} }
 $.  From Proposition \ref{closure-intersection} we also have $\wedge_Y = \wedge_{\scriptsize\Cl^{\mbox{\tiny{\texttt{ultra}}}}(Y)} $.
 
\begin{thm}\label{semistar-closure}
 \emph{(Finocchiaro-Fontana-Loper  \cite[Theorem 4.9]{fifolo-13b})}
Let $A$ be an integral domain, $K$  its quotient field, and $Y_1,Y_2$ two nonempty subsets of $\zar(K|A)$. Then, the following conditions are equivalent. 
\begin{enumerate}[\rm(i)]
\item The semistar operations of finite type associated to $\wedge_{Y_1}$ and $\wedge_{Y_2}$ are the same, that is, $(\wedge_{Y_1})_f=(\wedge_{Y_2})_f$. 
\item The subsets $\ad^{\mbox{\tiny{\texttt{ultra}}}}(Y_1),\, \ad^{\mbox{\tiny{\texttt{ultra}}}}(Y_2)$ of \ $\zar(K|A)$ have the same Zariski--generic closure, that is,  $\ad^{\mbox{\tiny{\texttt{ultra}}}}(Y_1)^
\downharpoonright=\ad^{\mbox{\tiny{\texttt{ultra}}}}
(Y_2)^\downharpoonright$. 
\end{enumerate}
\end{thm}

Let $A$ be an integral domain, $K$ its quotient field and $Z:=\zar(K|A)$. For any nonempty subset $Y \subseteq Z$, consider the $K-$function ring 
$$
\Kr(Y):=\bigcap \{V(T)\mid V\in Y\}\,.
$$
Note that, if we consider on the integral domain $A$ the valuative (e.a.b.) semistar operation $\wedge_Y$ defined above, then the Kronecker function ring associated to  $\wedge_Y$, $\Kr(A, \wedge_Y)$, coincides with  $\Kr(Y)$ \cite[Corollary 3.8]{folo-01a}.

 We say that $A$ is \emph{a vacant domain} if it is integrally closed and, for any representation $Y$ of $A$ (i.e., $A = \bigcap \{V\mid V\in Y\}$), we have $\Kr(Y)=\Kr(Z)$; for instance, a Pr\"ufer domain is vacant (see \cite{fa-10}).

\begin{cor} \label{vacant}
Let $A$ be an integrally closed domain and $K$  its quotient field. The following conditions are equivalent.
\begin{enumerate}[\rm(i)]
\item $A$ is a vacant domain.
\item For any representation $Y$ of $A$, $\ad^{\mbox{\tiny{\texttt{ultra}}}}
(Y)^\downharpoonright=
 \zar(K|A)$. 
\end{enumerate}
\end{cor}

Keeping in mind that it is known that the ultrafilter topology  and the constructible topology on $\zar(K|A)$ coincide (Theorem \ref{zar-ultra}(3)),  the following result follows easily from Corollary \ref{R-inv}(2). 

\begin{prop}
Let $K$ be a field and $A$ be a subring of $K$. For any subset $Y$ of $\zar(K|A)$,  $\ad^{\mbox{\tiny{\texttt{inv}}}}(Y)=\ad^{R}(\ad^{\mbox{\tiny{\texttt{ultra}}}}(Y))$.
\end{prop}

From the previous proposition, we can restate Corollary \ref{vacant} as follows: $A$ is a vacant domain if and only if
for any representation $Y$ of $A$, $\ad^{\mbox{\tiny{\texttt{inv}}}}
(Y)=\zar(K|A)$.

\bigskip

Recall that a \emph{semistar operation} is \emph{complete} if it is \texttt{e.a.b.} and of finite type. 
 In order to state some characterizations of the complete semistar operations, we need some terminology.

For a domain $A$  and a semistar operation $\star$ on $A$, we say that a valuation
overring $V$ of $A$ is a {\it $\star$-valuation overring of $A$} provided  $F^\star
\subseteq FV$, for each finitely generated $A$-module $F$ contained in the quotient field $K$ of $A$.
Set $\boldsymbol{\mathcal V}(\star) := \{V \mid V \mbox{ is a $\star$-valuation overring of } A\}$ and let $ b(\star) :=
\wedge_{\boldsymbol{\mathcal V}(\star)}$. Finally, if $\star$ is an  \texttt{e.a.b.} semistar operation on $A$, we can consider the Kronecker function ring $\Kr(A, \star):= \bigcap \{V(T) \mid V \in \boldsymbol{\mathcal V}(\star)\}= \Kr(\boldsymbol{\mathcal V}(\star)) $ \cite[Theorem 14(3)]{folo-06} and we can define a semistar operation $\kr(\star)$ on $A$ by setting $E^{\footnotesize \kr(\star)} := E\Kr(A, \star) \cap K$ for each nonzero $A$-module $E$ contained in $K$ (see for instance, \cite{folo-09}). Then,

\begin{prop}\label{complete}
 \emph{(Fontana-Loper  \cite[Proposition 3.4]{folo-01b}, \cite[Corollary 5.2]{folo-01a}, \cite[Proposition 6.3]{folo-07}, \cite{folo-09})}
 Given a semistar operation $\star$, the following are equivalent:
\begin{enumerate} 
 \item[\rm (i)] $\star$ is complete.
  \item[\rm (ii)] $\star = b(\star)$.
   \item[\rm (iii)] $\star = \kr(\star)$.
   \end{enumerate}

\end{prop}

The following result provides a topological characterization of when a semistar operation is complete.

\begin{thm} \label{complete2} 
 \emph{(Finocchiaro-Fontana-Loper  \cite[Theorem 4.13]{fifolo-13b})} 
Let $A$ be an integral domain, $K$  its quotient field and $\star$  a semistar operation on $A$. Then, the following conditions are equivalent. 
\begin{enumerate}[\rm(i)]
\item $\star$ is complete.
\item There exists a closed subset { $Y$} of $\zar(K|A)^{\mbox{\tiny{\texttt{cons}}}}$  such that {$Y = {Y}^\downharpoonright$ and $\star=\wedge_{Y}$.}

\item There exists a compact  subspace { $Y'$} in $\zar(K|A)^{\mbox{\tiny{\texttt{cons}}}}$  such that { $\star=\wedge_{Y'}$}.

\item There exists a { quasi-}compact { subspace of} $Y''$ of $\zar(K|A)^{\mbox{\tiny{\texttt{zar}}}}$  such that  $\star=\wedge_{Y''}$.
\end{enumerate}
\end{thm}

From Theorems \ref{semistar-closure} and \ref{complete2}  we easily deduce the following:

\begin{cor}
Let $A$ be an integral domain, $K$  its quotient field, and $Y$  a nonempty subset of  \ $\zar(K|A)$. Then $(\wedge_Y)_f=\wedge_{\scriptsize\ad^{\mbox{\tiny{\texttt{cons}}}}(Y)}$.
\end{cor}

Finally,  we can formulate some of the previous results in terms of Hochster's inverse topology   \cite[Theorem 4.9, Corollary 4.10]{fifolo-13b}.

\begin{cor} \label{inv-wedge}
Let $A$ be an integral domain, $K$ be its quotient field. The following statements hold.
\begin{enumerate}[\rm(a)]
\item If $Y_1,Y_2$ are nonempty subsets of $\zar(K|A)$, then  the following are equivalent.

\begin{enumerate}[\rm(i)]
\item $(\wedge_{Y_1})_f = (\wedge_{Y_2})_f$.
\item $\Kr(Y_1) = \Kr(Y_2)$.
 \item $\Cl^{\mbox{\tiny{\texttt{inv}}}}(Y_1) = \Cl^{\mbox{\tiny{\texttt{inv}}}}(Y_2)$.
 \end{enumerate}

\item $A$ is a vacant domain if and only if it is integrally closed and any representation $Y$ of $A$ is dense in $\zar(K|A)$ with respect to the inverse topology. 
\item For any nonempty subset $Y$ of $\zar(K|A)$,  $(\wedge_Y)_f=\wedge_{\scriptsize\ad^{
\mbox{\tiny{\texttt{inv}}}}(Y)}$.
\end{enumerate}
\end{cor}

\medskip

B. Olberding in \cite{ol-13} calls a subset $Y$ of $\zar(K|A)$ an {\it affine subset of $\zar(K|A)$} if $A^{\wedge_Y} := \bigcap \{ V \mid V \in Y \}$ is a Pr\"ufer domain with quotient field equal to $K$.
Note that $  Z :=\zar(K|A)$, equipped with the Zariski topology,  can be viewed as a locally ringed space with the structure sheaf defined by
$$ \mathcal O_{Z}(U) :=A^{\wedge_U} = \bigcap \{ V \mid V \in U \}\,, \; \mbox{ for each nonempty open subset $U$ of } Z \,,$$ (for more details, see \cite{ol-13}).
With this structure of locally ringed space, an affine subset $Y$ of $Z$ is not necessarily itself an affine scheme, that is 
$Y$ (endowed with the Zariski topology induced by $Z$) is not necessarily homeomorphic to $\Spec(\mathcal O_{Z}(Y))$, however, by Corollary \ref{inv-wedge}(a), is an inverse-dense subspace of the affine scheme \newline
 $(\Cl^{\mbox{\tiny{\texttt{inv}}}}(Y), {\mathcal O}_{Z}|_{_{\!\scriptsize \ad^{\mbox{\tiny{\texttt{inv}}}}(Y)}})$.

If $T$ is an indeterminate over $K$ and $A(T)$ is the Nagata ring associated to $A$ \cite[Section 33]{gi-72}, for each $Y \subseteq Z$, we can consider
$$Y(T) := \{ V(T) \mid V \in Y\} \subseteq \zar_0(K(T)|A(T)):= \{ V(T) \mid V \in Z\}\,,$$
 and, as above, 
  $\Kr(Y) = \bigcap \{ V(T) \mid V \in Y\}$.
  
  \bigskip
  
  The following statements were proved by Olberding \cite[Propositions 5.6 and 5.10 and Corollaries 5.7 and 5.8]{ol-13}. We give next a proof based on some of the results contained in \cite{fifolo-13b} and recalled above. 
  
\begin{prop} \label{prop-olb}
Let $Y$ be a subset of $Z:= \zar(K|A)$. Then,
\begin{enumerate}
\item[\rm{(1)}]  Assume that $A$ is a Pr\"ufer domain with quotient field $K$. Then,  $Y = \Cl^{\mbox{\tiny{\texttt{inv}}}}(Y)$  if and only if  \ $Y = \zar(K|R)$ for some overring $R$ of $A$. 
\item[\rm{(2)}]  $Y = \Cl^{\mbox{\tiny{\texttt{inv}}}}(Y)$ if and only if $Y(T) = \{W(T) \mid W \in Z \mbox{ and } W(T) \supseteq \Kr(Y)\}$.
\item[\rm{(3)}] $\Cl^{\mbox{\tiny{\texttt{inv}}}}(Y)= \{W \in Z \mid   W(T) \supseteq \Kr(Y)\}$,  hence $(\Cl^{\mbox{\tiny{\texttt{inv}}}}(Y))(T) = \zar(K(T)|\Kr(Y))$.
\item[\rm{(4)}] $A^{\wedge_Y} = \bigcap \{W \in \Cl^{\mbox{\tiny{\texttt{inv}}}}(Y) \}$.
\item[\rm{(5)}] If $Y$ is an affine subset of $Z$ then 
$$\Cl^{\mbox{\tiny{\texttt{inv}}}}(Y) =\{W \in Z \mid   W \supseteq A^{\wedge_Y} \} = \zar(K|A^{\wedge_Y})\,.$$
\item[\rm{(6)}] Assume that  $Y_1$ and $Y_2$ are two affine subsets of $Z$, then
$$ \bigcap\{V\mid V\in Y_1\}=\bigcap \{V\mid V\in Y_2\} \; \Leftrightarrow \; \Cl^{\mbox{\tiny{\texttt{inv}}}}(Y_1)=
\Cl^{\mbox{\tiny{\texttt{inv}}}}(Y_2)\,.$$

\item[\rm{(7)}] $(\Cl^{\mbox{\tiny{\texttt{inv}}}}(Y))(T) =\Cl^{\mbox{\tiny{\texttt{inv}}}}(Y(T))$.
\item[\rm{(8)}] Each ring of fractions of $A^{\wedge_Y}$ can be represented as an intersection of valuation domains contained in a subset of $\Cl^{\mbox{\tiny{\texttt{inv}}}}(Y)$, in other words, if $S$ is a multiplicatively closed subset of $A^{\wedge_Y}$ then $(A^{\wedge_Y})_{S} = A^{\wedge_{\Sigma}}$ for some $\Sigma \subseteq \Cl^{\mbox{\tiny{\texttt{inv}}}}(Y)$.

\item[\rm{(9)}] The  canonical homeomorphism of topological spaces (all endowed with the Zariski topology)
$$ \tau: \Spec(\Kr(K|A)) \rightarrow \zar(K|A)\,,\; Q \mapsto \Kr(K|A)_Q \cap K $$
(where $\tau = \sigma^{-1}$, see Theorem \ref{zar-spectral}(4)) determines a continuous injective map $\Spec(\Kr(Y)) \rightarrow \zar(K|A)$ which restricts to a homeomorphism of 
 $\Spec(\Kr(Y))$ (respectively, $\Max(\Kr(Y))$) onto $ \Cl^{\mbox{\tiny{\texttt{inv}}}}(Y)$ (respectively, $\Max(\Cl^{\mbox{\tiny{\texttt{inv}}}}(Y))$).  
\end{enumerate}
\end{prop}
\begin{proof}  
 (1) 
Let $Y $ be a non empty, closed set with respect to the inverse topology, and let $ R:=A^{\wedge_Y}:=\bigcap \{V \mid V\in Y\}$. 
 Since $R$ is an overring of the Pr\"ufer domain $A$, $R$ is also a Pr\"ufer domain, thus it is vacant.
  By Corollary \ref{inv-wedge}(b), $Y$ is a dense subspace of $\zar(K|R)$, with the inverse topology, i.e. $\Cl^{\mbox{\tiny{\texttt{inv}}}}(Y) \cap \zar(K|R)=\zar(K|R)$. Thus $\zar(K|R)\subseteq \Cl^{\mbox{\tiny{\texttt{inv}}}}(Y)$. On the other hand, the inclusion $Y\subseteq \zar(K|R)$ implies $\Cl^{\mbox{\tiny{\texttt{inv}}}}(Y) \subseteq \zar(K|R)$, since $\zar(K|R)$ is clearly inverse closed. Therefore,  $Y=\Cl^{\mbox{\tiny{\texttt{inv}}}}(Y)=\zar(K|R)$. The converse holds for any integral domain.

(2)  Let $A(T)$ be Nagata ring associated to $A$. By Theorem \ref{zar-spectral}(3), the natural map $\varphi: \zar(K(T)|\Kr(K|A)) \subseteq \zar(K(T)|A(T))\rightarrow \zar(K|A)$, $W\mapsto W\cap K$, is a homeomorphism with respect  to the  Zariski topology  and to the constructible topology and thus also with respect to the inverse topology.   By the previous homeo\-morphism $Y$ is inverse-closed in  $\zar(K|A)$ if and only if $Y(T)$ is inverse-closed in $\zar(K(T)|A(T))$. Therefore, 
 if $Y$ is inverse-closed, then $Y(T)$ is inverse-closed in  $\zar(K(T)|A(T))$  and thus, by (1),  $Y(T) = \zar(K(T)|\Kr(Y))$. Finally, by Proposition \ref{zar0}, 
 $\zar_0(K(T)|\Kr(Y))= \zar(K(T)|\Kr(Y))$.  
 
 Conversely, if  $Y(T) = $ $\zar_0(K(T)|\Kr(Y)) $, then  $Y(T) = \zar(K(T)|\Kr(Y))$ (Pro\-position \ref{zar0}), hence by (1) $Y(T)$ is inverse-closed.
 
 (3) By (1), $\zar(K(T)|\Kr(Y)) =Y(T)$ is an inverse-closed subspace of  the space $\zar(K(T)|A(T))$, then $\{W \in Z \mid $ $  W(T) \supseteq \Kr(Y)\}$ is an inverse-closed subspace of $\zar(K|A)$ (Theorem \ref{zar-spectral}(3)) and it obviously contains $Y$. Let $Y'$ be an inverse-closed subspace of $\zar(K|A)$ containing $Y$, then clearly $Y(T) \subseteq Y'(T) = \zar(K(T)|\Kr(Y'))$ and so  $\{W \in Z \mid $ $  W(T) \supseteq \Kr(Y)\} \subseteq \{W' \in Z \mid $ $  W'(T) \supseteq \Kr(Y')\} = \Cl^{\mbox{\tiny{\texttt{inv}}}}(Y') = Y'$. The last part of the statement follows from Theorem \ref{zar-spectral}(3).
 
 (4) is a straightforward consequence of Corollary \ref{inv-wedge}(c).
 
 (5)  Since $Y$ is an affine set, $A^{\wedge_Y}$ is a Pr\"ufer domain with quotient field $K$. Therefore, by (1),
 $ \Cl^{\mbox{\tiny{\texttt{inv}}}}(Y)= \zar(K|R)$ for some  overring $R$ of $A$ that, without loss of generality, we can assume integrally closed.
 Hence, $A^{\wedge_Y} = A^{\wedge_{\tiny{\Cl^{\mbox{\tiny{\texttt{inv}}}}(Y)}}}= \bigcap \{ V \in \zar(K|R)\} = R$ 
 and so
 $\Cl^{\mbox{\tiny{\texttt{inv}}}}(Y)= \zar(K|A^{\wedge_Y})$.
 
 (6) The implication $(\Leftarrow)$ holds in general and  is a straightforward consequence of  Corollary \ref{inv-wedge}(a).  For $(\Rightarrow)$, assume more generally that $ R:= \bigcap\{V\mid V\in Y_1\}=\bigcap \{V\mid V\in Y_2\}$ is a vacant domain with quotient field $K$ then, by Corollary \ref{vacant} (or Corollary \ref{inv-wedge}(b)), $\Cl^{\mbox{\tiny{\texttt{inv}}}}(Y_1) = \zar(K|R) =\Cl^{\mbox{\tiny{\texttt{inv}}}}(Y_2)$. 
 
 (7) Since $\Cl^{\mbox{\tiny{\texttt{inv}}}}(Y) = \Cl^{\mbox{\tiny{\texttt{inv}}}}(\Cl^{\mbox{\tiny{\texttt{inv}}}}(Y))$ then, by (2),  we have $(\Cl^{\mbox{\tiny{\texttt{inv}}}}(Y))(T) = $ $\zar(K(T)|\Kr(Y))$. The conclusion follows from (3).
 
 (8)  Note that $A^{\wedge_Y} = \Kr(Y) \cap K$ and so  $(A^{\wedge_Y})_S = \Kr(Y)_S \cap K$. Since each overring of  $\Kr(Y)$ is a $K$-function ring, there exists $\Sigma \subseteq Y$ such that  $ \Kr(Y)_S = \Kr(\Sigma)$ (Proposition \ref{basic}(1)).  We conclude that  $(A^{\wedge_Y})_S = \Kr(\Sigma) \cap K = A^{\wedge_\Sigma}$.
 
 (9) Observe that $\Kr(Y)$ is an overring of the Pr\"ufer domain $\Kr(K|A)$. Thus $\Spec(\Kr(Y))$ is canonically embedded in $\Spec(\Kr(K|A))$. If $P \in \Spec(\Kr(Y))$ then,  by (3), $\Kr(Y)_P \cap K \in \Cl^{\mbox{\tiny{\texttt{inv}}}}(Y)$. The conclusion follows from Theorem \ref{zar-spectral}(4).  
  \end{proof}

\begin{rk}  \emph{Another proof of Proposition \ref{prop-olb}(1) is based on the fact that when $A$ is a Pr\"ufer domain, it is easy to see that $\zar(K|S\cap T)= \zar(K|S) \cup \zar(K|T)$ for each pair of overrings $S$ and $T$ of $A$.  Now, suppose  $Y = \Cl^{\mbox{\tiny{\texttt{inv}}}}(Y)$.  When $A$ is Pr\"ufer, $ \Cl^{\mbox{\tiny{\texttt{inv}}}}(Y) = \bigcap \{\zar(K|A_\lambda) \mid \lambda \in \Lambda \}$,  where $ A_\lambda $ is a finitely generated overring of  $A$. Moreover,  $\bigcap \{\zar(K|A_\lambda) \mid \lambda \in \Lambda \} = \zar(K|R)$, where  $R$ is the ring generated by $ \bigcup \{A_\lambda\mid \lambda \in \Lambda\}$. Conversely, if $Y = \zar(K|R)$ for some overring $R$ of $A$, then $Y =  \bigcap \{\zar(K|A[r]) \mid r\in R \}$ and thus  is inverse-closed since it is the intersection of a family of inverse-closed subsets, $\zar(K|A[r])$, of $\zar(K|A)$. }

\end{rk}

\begin{thm} \emph{(Olberding  \cite[Proposition 5.10]{ol-13})} Let $K$ be a field and $A$ a subring of $K$. Let $R$ be an integrally closed domain with quotient field $K$, containing as subring $A$.
Given a subset $Y \subseteq \zar(K|A)$ such that $R = \bigcap \{V\mid V\in Y\}$
 and $Y= \Cl^{\mbox{\tiny{\texttt{inv}}}}(Y)$,
  then $\Phi(Y):=\Kr(Y)$ is a $K$--function ring such that
   $\Phi(Y) \cap K = R$. 
   Conversely, given a $K$--function ring $\Psi$, $A(T) \subseteq \Psi \subseteq K(T)$ such that $\Psi \cap K = R$, then $\mathsf F(\Psi) := \{W \cap K \mid W \in \zar(K(T)|\Psi)\}$ is an inverse-closed subspace of $\zar(K|A)$ and $R =  \bigcap \{V\mid V\in \mathsf F(\Psi) \}$. Furthermore, $\mathsf F(\Phi(Y)) =Y$ for each inverse-closed subspace  $Y$ of $\zar(K|A)$ and
   $\Phi(\mathsf F(\Psi)) =\Psi$, for each $K$--function ring $\Psi$, $A(T) \subseteq \Psi \subseteq K(T)$.   
   \end{thm}

From the previous theorem, it follows that if $Y_1$ and $Y_2$ are two different subsets of $Z :=\zar(K|A)$ such that 
$\Cl^{\mbox{\tiny{\texttt{inv}}}}(Y_1) = \Cl^{\mbox{\tiny{\texttt{inv}}}}(Y_2)$ then $\Kr(Y_1)$ and $\Kr(Y_2)$ are two different $K$--function rings such that $\Kr(Y_1) \cap K = \bigcap \{ V  \mid V \in Y_1 \} = \bigcap \{ V  \mid V \in Y_2 \} = \Kr(Y_2) \cap K$. Furthermore, if $R$ is an integrally closed domain with quotient field $K$ then $\zar(K|R)$ is an inverse-closed subspace of $\zar(K|A)$, $R =  \bigcap \{ V  \mid V \in \zar(K|R)\}$ and $\Kr(\zar(K|R))$ is the smallest
$K$--function ring such that $\Kr(\zar(K|R)) \cap K =R$. If we assume that $R$ is a Pr\"ufer domain, then $\Kr(\zar(K|R))$ is the unique $K$--function ring such that $\Kr(\zar(K|R)) \cap K =R$ \cite[Theorem 32.15 and Proposition 32.18]{gi-72}.
\medskip

We have already observed that $Z:=\zar(K|A)$ (endowed with the Zariski topology) is always a spectral space, being canonically homeomorphic to $\Spec(\Kr(K|A))$. It is natural to investigate when the ringed space $(Z,\mathcal O_Z)$
is an affine scheme. 

\begin{thm}
\emph{(Olberding \cite[Theorem 6.1 and Corollaries 6.2 and 6.3]{ol-13})} Let $K$ be a field, $A$ a subring of $K$ and  $Y$ a subspace of $Z:=\zar(K|A)$ (endowed with the Zariski topology). Then,
\begin{enumerate}
\item[\rm{(1)}]  
$(Y, \mathcal O_Y)$ 
is an affine scheme if and only if
 $\mathcal O_Y(Y)$ is a Pr\"ufer domain and
  $Y =\Cl^{\mbox{\tiny{\texttt{inv}}}}(Y)$ or, equivalently, if and only if $Y$ is an inverse-closed affine subset of $Z$.
\item[\rm{(2)}] $(Z, \mathcal O_Z)$ 
is an affine scheme if and only if
the integral closure of $A$ in $K$  is a Pr\"ufer domain with quotient field $K$.
\item[\rm{(3)}] $Y =\Cl^{\mbox{\tiny{\texttt{inv}}}}(Y)$ if and only if $(Y(T), \mathcal O_{Y(T)})$ is an affine scheme.
 \end{enumerate}

\end{thm}

\noindent  {\bf Acknowledgment.} The authors thank  the referee for providing helpful suggestions and pointing out to them the very recent paper by N. Schwartz \cite{sch-13}.

\end{document}